\documentclass[11pt]{amsart}
\usepackage[utf8]{inputenc}

\usepackage{hyperref}
\usepackage{lscape}
\usepackage{amsmath,amssymb,amsfonts,amsthm}
\usepackage[margin=1in]{geometry}
\usepackage{wasysym}
\usepackage{youngtab}
\usepackage{ytableau}
\usepackage[backend=bibtex8]{biblatex}
\usepackage{comment}

\newtheorem{thm}{Theorem}

\numberwithin{thm}{section}
\newtheorem{lemma}[thm]{Lemma}

\newtheorem{prop}[thm]{Proposition}
\newtheorem{open}[thm]{Problem}

\theoremstyle{definition}
\newtheorem{definition}[thm]{Definition}
\newtheorem{example}[thm]{Example}
\newtheorem{remark}[thm]{Remark}

\DeclareMathOperator{\cc}{cc}

\DeclareMathOperator{\flip}{flip}
\DeclareMathOperator{\Frob}{Frob}

\DeclareMathOperator{\rev}{rev}
\DeclareMathOperator{\SSYT}{SSYT}
\newcommand{\SYT}{\mathrm{SYT}}

\newcommand{\ccind}{\mathrm{ccTab}}
\newcommand{\sh}{\mathrm{sh}}
\newcommand{\sgn}{\mathrm{sgn}}
\newcommand{\maj}{\mathrm{maj}}
\newcommand{\comaj}{\mathrm{comaj}}

\newcommand{\DR}{\mathrm{DR}}

\newcommand{\CC}{\mathbb{C}}
\newcommand{\QQ}{\mathbb{Q}}

\newcommand{\Tab}{\mathrm{Tab}}
\newcommand{\Des}{\mathrm{Des}}
\newcommand{\des}{\mathrm{des}}

\newcommand{\grFrob}{\mathrm{grFrob}}

\newtheorem*{MainTheorem}{Theorem \ref{thm:main}}

\title[Higher Specht polynomials under the diagonal action]{Higher Specht polynomials under the diagonal action}
\author{Maria Gillespie}
\address{Maria Gillespie \\ Department of Mathematics \\Colorado State University \\ Fort Collins, CO 80523 \\ United States of America}
\email{\href{mailto:Maria.Gillespie@colostate.edu}{Maria.Gillespie@colostate.edu}}
\thanks{The author was partially supported by NSF DMS award number 2054391.}
\date{\today}

\begin{document}

\begin{abstract}
 We introduce higher Specht polynomials - analogs of Specht polynomials in higher degrees - in two sets of variables $x_1,\ldots,x_n$ and $y_1,\ldots,y_n$ under the diagonal action of the symmetric group $S_n$.  This generalizes the classical Specht polynomial construction in one set of variables, as well as the higher Specht basis for the coinvariant ring $R_n$ due to Ariki, Terasoma, and Yamada, which has the advantage of respecting the decomposition into irreducibles.
 
 As our main application of the general theory, we provide a higher Specht basis for the hook shape Garsia--Haiman modules.  In the process, we obtain a new formula for their doubly graded Frobenius series in terms of new generalized cocharge statistics on tableaux.
\end{abstract}

\maketitle{}

\section{Introduction}

The \textit{diagonal action} of the symmetric group $S_n$ on the polynomial ring $$\mathbb{C}[\mathbf{x},\mathbf{y}]:=\mathbb{C}[x_1,\ldots,x_n,y_1,\ldots,y_n],$$ defined by $$\pi\cdot  f(x_1,\ldots,x_n,y_1,\ldots,y_n)=f(x_{\pi(1)},\ldots,x_{\pi(n)},y_{\pi(1)},\ldots,y_{\pi(n)}),$$ has been the subject of much study in modern algebraic combinatorics.  The ring of \textbf{diagonal coinvariants}, which we call $\DR_n$ in this paper, is the quotient of $\mathbb{C}[\mathbf{x},\mathbf{y}]$ by the ideal generated by the $S_n$-invariants with no constant term.  It arises naturally in the geometry of the Hilbert scheme of $n$ points in the plane $\CC^2$.  Haiman  \cite{HaimanDR} used this connection to prove the famous \textit{$(n+1)^{n-1}$ Conjecture}, which states that $\dim_\CC(\DR_n)=(n+1)^{n-1}$ as a complex vector space.  

Haiman used similar methods to prove the \textit{$n!$ conjecture} \cite{HaimanGeometry}.  This states that the \textit{Garsia--Haiman modules} $\DR_\mu$ (which are also quotients of $\mathbb{C}[\mathbf{x},\mathbf{y}]$ describing some of the local information of a limit as the $n$ points all approach $0$ in the plane) have dimension $n!$ for any partition $\mu$ of size $n$.  The proof of the $n!$ conjecture was the crucial step in the proof of the Macdonald Positivity Conjecture, which states that the transformed \textit{Macdonald polynomials} \cite{Macdonald} have a positive expansion in terms of the Schur symmetric functions.

Despite these advancements, it remains open to understand the $n!$ and $(n+1)^{n-1}$ conjectures from a more combinatorial standpoint, in the following sense.

\begin{open}\label{prob:nfact}
 Find an explicit basis of $n!$ polynomials for $\DR_\mu$, where $\mu$ is a partition of $n$.
\end{open}

\begin{open}\label{prob:basis-DR}
 Find an explicit basis of $(n+1)^{n-1}$ polynomials for $\DR_n$.
\end{open}

Problem \ref{prob:nfact} is open for general partitions $\mu$, while Problem \ref{prob:basis-DR} has very recently been addressed by Carlsson and Oblomkov \cite{CarlssonOblomkov}, who gave a construction of a basis for $\mathrm{DR}_n$ by establishing connections with affine Schubert calculus (though it is not a higher Specht basis as defined below). 

In this paper we explore a new approach to understanding and finding bases for these quotients under the diagonal action via the method of \textit{higher Specht polynomials}.  The standard construction of the \textit{Specht modules} $V_\lambda$, which are the irreducible representations of $S_n$ where $\lambda$ ranges over all partitions of $n$, is often presented via \textit{Young tabloids} (see \cite{Sagan} for a thorough introduction to the Young tabloid construction).  This construction is equivalent to defining a submodule of $\mathbb{C}[x_1,\ldots,x_n]$ spanned by \textit{Specht polynomials} as follows.

\begin{definition}
  The \textbf{Specht polynomial} $F_T$ corresponding to a Young tableau $T$ whose entries are $1,2,3,\ldots,n$ is given by $$F_T=\prod_{C\in\mathrm{col}(T)}\prod_{\substack{i,j\in C\\ i\text{ above }j}}(x_i-x_j)$$ where $\mathrm{col}(T)$ is the set of columns of $T$. (See Figure \ref{fig:Specht}.)
\end{definition} 

\begin{figure}[b]
    \centering
   $T=\raisebox{-0.5cm}{\young(3,716,2845)}$\hspace{2cm} $F_T=(x_3-x_7)(x_7-x_2)(x_3-x_2)\cdot (x_1-x_8)\cdot (x_6-x_4)$
    \caption{A tableau $T\in \Tab(4,3,1)$ in French notation, and the Specht polynomial $F_T$.}
    \label{fig:Specht}
\end{figure}
Writing $\Tab(\lambda)$ for the set of all fillings of the boxes of the Young diagram of $\lambda$ with $1,2,\ldots,n$ in some order, one can then define the Specht module 
as $$M_\lambda=\mathrm{span}\{F_T: T\in \Tab(\lambda)\}\subseteq \mathbb{C}[x_1,\ldots,x_n],$$ and we have $M_\lambda\cong V_\lambda$ with a basis given by the $F_T$ such that $T$ is a \textbf{standard Young tableau (SYT)}, in which the entries are increasing along rows and up columns.  The \textbf{Garnir relations}, described in Section \ref{sec:general} below, give a straightening algorithm for expressing any $F_T$ in terms of the standard Young tableau basis, and thereby gives a rule for computing with the $S_n$ action on the Specht module.

 In \cite{ATY}, Ariki, Terasoma, and Yamada noted that $F_T$ may also be defined (up to a constant) as a Young idempotent operator applied to a monomial.  In particular, define $$\varepsilon_T=\sum_{\tau \in C(T)}\sum_{\sigma \in R(T)}\sgn(\tau) \tau \sigma$$ where $C(T)\subseteq S_n$ is the group of \textit{column permutations} that preserve the columns of $T$, and $R(T)\subseteq S_n$ is the group of \textit{row permutations}. 
 Then it is not hard to check that $F_T=\varepsilon_T(x_T^r)$ where $x_T^r=\prod x_i^{\mathrm{row}(i)-1}$ with $\mathrm{row}(i)$ denoting the row that $i$ occurs in in T, indexed from bottom to top.  Ariki, Terasoma, and Yamada then generalized this construction to define \begin{equation}\label{eq:ATY} F_T^S = \varepsilon_T(x_T^S)\end{equation} where $T,S$ are a pair of standard Young tableaux of the same shape, and $x_T^S$ is a monomial whose subscripts are given by the entries of $T$ and whose exponents are determined by the \textit{cocharge} algorithm on the corresponding boxes in $S$.  Their main result was as follows.

 \begin{thm}[\cite{ATY}]
     The polynomials $F_T^S$ form a basis for the one-variable coinvariant ring $R_n=\mathbb{C}[x_1,\ldots,x_n]/(e_1,\ldots,e_n)$, where $e_i$ is the $i$-th elementary symmetric polynomial.
 \end{thm}

The importance of this basis is that it is a \textit{higher Specht basis}, which respects the decomposition into irreducible $S_n$ representations, and which we define presicely in this paper as follows.

\begin{definition}
    A \textbf{higher Specht basis} for a (graded) $S_n$ module $M$ is a basis $\mathbb{B}$ that admits a set partition $\bigsqcup \mathbb{B}_\lambda$ such that:
    \begin{enumerate}
        \item Each $\mathbb{B}_\lambda$ spans a copy of an irreducible representation $V_\lambda$ in the decomposition of $M$ into irreducibles,
        \item There is a bijection from $\mathbb{B}_\lambda$ to the set of ordinary Specht polynomials $F_T$ for shape $\lambda$, that preserves the $S_n$ action with respect to each basis.
    \end{enumerate} 
\end{definition}

 The terminology ``higher Specht'' in this setting is also due to the fact that the polynomials $F_T^S$ are in general higher-degree analogs of the ordinary Specht polynomials.

 In \cite{GillespieRhoades}, the author and Rhoades found higher Specht bases for both the modules $R_{n,k}$ (the \textit{Haglund-Rhoades-Shimozono modules} defined in the context of the $t=0$ specialization of the Delta conjecture \cite{HRS}) and $R_\mu$ (the \textit{Garsia--Procesi modules}, which are the cohomology rings of the fibers of the Springer resolution \cite{DeConciniProcesi,GarsiaProcesi,Tanisaki}).  They proved their construction was a basis in the former case and conjectured for the latter, proving it for $\mu$ having two parts.  
 
 This construction was similar to that of \cite{AllenOneVariable}, in which Allen constructed a basis that respects the decomposition into irreducibles for $R_n$, and for $R_\mu$ for $\mu$ having two parts or being a hook shape.  (Allen's bases are not higher Specht bases in the strongest sense, since the $S_n$ action on the basis was different).  In \cite{AllenDescentMonomials} and \cite{AllenBitableaux1}, Allen also began an exploration the two-variable case, focusing on the diagonally symmetric subring of the polynomial ring in two variables and its quotients.  

 In this paper, we establish a general theory for constructing higher Specht polynomials in $\mathbb{C}[\mathbf{x},\mathbf{y}]$ under the diagonal action.  In particular, we show that for certain conditions on the sequences of exponents $c=(c_1,\ldots,c_n)$ and $d=(d_1,\ldots,d_n)$, the set of polynomials $$F_T^{c,d}:=\varepsilon_T (x_1^{c_1}\cdots x_n^{c_n}y_1^{d_1}\cdots y_n^{d_n}),$$ for all standard Young tableaux $T$ of a fixed shape $\lambda$, form a higher Specht basis for a copy of the irreducible Specht module $V_\lambda$.  We then use our theory to develop our main result, a higher Specht basis for the Garsia--Haiman modules for hook shapes:

 \begin{thm}\label{thm:main}
     Suppose $\mu=(n-k+1,1^k)$ is a hook shape of height $k$ and $T,S$ denotes a pair of standard Young tableaux of the same shape $\lambda$.  Then there are modified cocharge labelings (see Section \ref{sec:hooks}) $\ccind_\mu(S)$ and $\ccind_\mu'(S)$ such that if $$\mathbf{xy}_T^S=\prod_{b\in \lambda}x^{\ccind_\mu(S)(b)}_{T(b)}y_{T(b)}^{\ccind'_\mu(S)(b)}\hspace{0.5cm}\text{ and }\hspace{0.5cm}F_T^S=\varepsilon_T(\mathbf{xy}_T^S),$$ then the set of polynomials $\{F_T^S\}$ ranging over all pairs $T,S$ of standard Young tableaux of the same shape forms a higher Specht basis for $\DR_\mu$.
 \end{thm}

We have verified this computationally on Sage up to $n=8$ as well, and we provide the proof in Section \ref{sec:hooks}.  The proof relies on some of the structural results for hook shape Garsia--Haiman modules established in \cite{HookBases} by Adin, Remmel, and Roichman.  In particular, we multiply the $F_T^S$ constructions by certain diagonally symmetric polynomials to construct a higher Specht basis for an intermediate quotient $\mathcal{P}_n^{(k)}$ between $\mathbb{C}[\mathbf{x},\mathbf{y}]$ and $\DR_\mu$.
  The paper is organized as follows.  In Section \ref{sec:background}, we establish notational and mathematical background on higher Specht bases and the Frobenius series of $\DR_\mu$, and also provide a bijective proof of a known equidistribution result on the major index and cocharge.  This does not to our knowledge directly appear in the literature as a bijective proof, so we include it for completeness, as we will use similar methods to construct our new generalized cocharge statistics.
  
  In Section \ref{sec:general}, we build the theory of higher Specht polynomials in two sets of variables, and show that many of the constructions from one set of variables apply to any number of sets of variables.  In Section \ref{sec:hooks}, we prove Theorem \ref{thm:main}.
  
  Finally, in Section \ref{sec:future}, we outline several possible directions for exploration that this paper opens up.  We also give example higher Specht constructions in special cases for the diagonal coinvariant rings, in particular for $\DR_2$ and $\DR_3$.

\subsection{Acknowledgments}

We thank Kelvin Chan, Bryan Gillespie, Sean Griffin, Jake Levinson, and Brendon Rhoades for helpful mathematical conversations pertaining to this project.  Special thanks to Jake Levinson for his support and encouragement that made it possible.

\section{Definitions and background}\label{sec:background}

\subsection{Tableaux and statistics}

The \textbf{Young diagram} of a partition $\lambda=(\lambda_1,\ldots,\lambda_k)$, written so that $\lambda_1\ge \lambda_2\ge \cdots$, is the left-justified partial grid of squares having $\lambda_i$ squares in the $i$-th row from the bottom (we use the ``French'' convention for diagrams in this paper).  We write $n=|\lambda|$ for the number of squares of $\lambda$, and $\ell(\lambda)=k$.

A \textbf{semistandard Young tableau}, or SSYT, of shape $\lambda$ is a way of filling each square of $\lambda$ with a positive integer such that the rows are weakly increasing left to right and the columns are strictly increasing from bottom to top.  In general we will also allow entries from any totally ordered set (not necessarily the positive integers), with the same definition.   For an SSYT $T$, we write $\sh(T)=\lambda$ if $T$ has shape $\lambda$.

The \textbf{reading word} of any tableau is the word formed by concatenating the rows from top to bottom (and this ordering of the squares is called \textbf{reading order}).  For instance, the reading word of the tableau at middle in Figure \ref{fig:tableaux} is $748251369$.

The \textbf{content} of a tableau or word is the tuple $(m_1,m_2,\ldots)$ where $m_i$ is the number of times $i$ appears in the tableau or word.  For instance, the tableau at left in Figure \ref{fig:tableaux} has content $(3,2,2,1,1)$.

An SSYT is \textbf{standard} if each of the entries $1,2,\ldots,n$ is used exactly once (where $n=|\lambda|$).  We write $\SYT(\lambda)$ for the set of standard Young tableaux of shape $\lambda$.  The \textbf{standardization} of an SSYT is the standard Young tableau formed by labeling the boxes $1,2,3,\ldots,n$ in a way that respects the ordering of the labels in the SSYT and breaks ties in reading order.  It is straightforward to verify that the standardization of an SSYT is an SYT.

We recall the well-known fact that permutations of $1,\ldots,n$ are in bijection with pairs $(P,Q)$ of SYT's with the same shape as each other, via the \textbf{RSK bijection}.  Thus there are $n!$ such pairs.  In particular, we will need the \textbf{RSK insertion} algorithm: given a word $w$, we define its \textbf{insertion tableau} $T$ by starting with the empty tableau, and inserting $w_1,w_2,\ldots$ inductively in the following manner.  Given a semistandard young tableau $S$, we insert the letter $a$ into the bottom row $r$, by either placing it at the end of row $r$ if it is greater than every element of $r$, or replacing entry $b$ of $r$ by $a$ where $b$ is the leftmost entry greater than $a$.  Then we insert $b$ into the second row in the same manner, and so on.  (See \cite{Fulton} for details and examples.)

A \textbf{standard general tableau} $T$ of shape $\lambda$ is a way of filling the squares of $\lambda$ with the numbers $1,2,\ldots,n$ each used once, with no restrictions on the rows or columns.  We write $\Tab(\lambda)$ for the set of all standard general tableaux of shape $\lambda$, and note that $|\Tab(\lambda)|=n!$.  (See Figure \ref{fig:tableaux}.)

\begin{figure}
    \centering
    \young(5,34,22,1113) \hspace{1cm} \young(7,48,25,1369) \hspace{1cm} \young(7,14,62,3985)
    \caption{From left to right: A semistandard Young tableau, a standard Young tableau, and a standard general tableau.} 
    \label{fig:tableaux}
\end{figure}

\begin{definition}
   A \textbf{descent} of a standard Young tableau $T$ is an entry $i$ such that $i+1$ occurs weakly to the left of (and necessarily above) $i$ in $T$.  We write $\Des(T)$ for the set of all descents of $T$ and $\des(T)$ for the number of descents.
   
   The \textbf{major index} of a standard Young tableau $T$ is defined as the sum of the descents: $$\maj(T)=\sum_{i\in \Des(T)} i.$$
\end{definition}

\begin{example}
    The middle tableau in Figure \ref{fig:tableaux} has descents $1$, $3$, and $6$.  Therefore, its major index is $1+3+6=10$.
\end{example}

It is well-known (due to Lusztig and Stanley \cite{Stanley}) that the number of copies of the irreducible representation $V_\lambda$ of $S_n$ in the coinvariant ring $R_n$ in degree $d$ is equal to the number of standard Young tableaux of shape $\lambda$ with major index $d$.  We will express this fact in terms of the graded Frobenius character in Section \ref{sec:Frobenius}. 

Lascoux and Sch\"utzenberger, in \cite{LascouxSchutzenberger} generalized the major index to a statistic on semistandard Young tableaux called \textbf{cocharge}.  They used it to obtain the graded decomposition of the Garsia--Procesi modules $R_\mu$ into irreducible $S_n$-modules.  We recall its definition here.

\begin{definition}
    The \textbf{cocharge labels} of a permutation $\pi$ are defined by labeling the $1$ with a subscript $0$, then searching leftwards cyclically for the $2,3,4,\ldots$, each time incrementing the subscript label unless the search wraps around the word.  The \textbf{cocharge} of $\pi$, written $\cc(\pi)$, is the sum of its cocharge labels.
\end{definition}

\begin{example}
    The permutation $25314$ has cocharge labels $2_15_23_11_04_1$, so $$\cc(25314)=1+2+1+0+1=5.$$

\end{example}

\begin{definition}
    For a word $w$ with partition content, the \textbf{first cocharge subword} $w^{(1)}$ is formed by searching from right to left for a $1$, then continuing left to search for a $2$, and so on, wrapping around cyclically if need be, and terminating at the largest entry.  The second cocharge subword $w^{(2)}$ is formed by repeating this process on the remaining letters, and so on.  Then its \textbf{cocharge} is $$\cc(w):=\sum_i \cc(w^{(i)}).$$
    Finally, the \textbf{cocharge} of a tableau is defined to be the cocharge of its reading word.
\end{definition}

\begin{example}
    We compute $\cc(433111222442311)$, by first labeling the first cocharge subword:
\[\mathbf{4}_3
{\color{gray}
3_{\phantom{1}}
}
\mathbf{3}_2
{\color{gray}
1_{\phantom{1}
}
1_{\phantom{1}}
1_{\phantom{1}}
2_{\phantom{1}}
2_{\phantom{1}}
2_{\phantom{1}}
4_{\phantom{1}}
4_{\phantom{1}}}
\mathbf{2}_1
{\color{gray}
3_{\phantom{1}}
1_{\phantom{1}}
}
\mathbf{1}_0.
\]
followed by the second:
\[
{\color{gray}
4_3
}
\mathbf{3}_2
{\color{gray}
3_2
1_{\phantom{1}}
1_{\phantom{1}}
1_{\phantom{1}}
2_{\phantom{1}}
2_{\phantom{1}}
}
\mathbf{2}_{1}
{\color{gray}
4_{\phantom{1}}
}
\mathbf{4}_{2}
{\color{gray}
2_1
3_{\phantom{1}}
}
{\bf 1}_{0}
{\color{gray}
1_0.
}
\]
and so on until we obtain the full labeling \[
4_3
3_2
3_2
1_0
1_0
1_0
2_0
2_{0}
2_{1}
4_{1}
4_{2}
2_1
3_0
1_0
1_0.
\]
Thus the cocharge is the sum of the subscripts, which is 12.    
\end{example}

It is known \cite{LascouxSchutzenberger} that the RSK insertion tableau of $w$ has the same cocharge as $w$.  In particular, cocharge is preserved under \textbf{Knuth equivalence} of words (which is also preserved by RSK), defined as the reachability of one word from another via \textit{elementary Knuth moves} of the form $$bac \leftrightarrow bca \hspace{1cm} \text{or} \hspace{1cm} acb\leftrightarrow cab$$
where $a< b\le c$ or $a\le b<c$ respectively, and the three entries are consecutive in the word.
 
\subsection{A weight-preserving bijection}

We now give a bijection $\phi:SYT(\lambda)\to \SYT(\lambda)$ (for any shape $\lambda$) that sends $\cc$ to $\maj$ and preserves the number of descents, and that will come up in our work on hook shapes in Section \ref{sec:hooks}.  We will require the following two operations on words.

\begin{definition}
    The \textbf{flip} of a permutation $\pi\in S_n$, written $\flip(\pi)$, is the permutation formed by replacing $\pi_i$ with $n+1-\pi_i$ for each entry $\pi_i$.

    The \textbf{reverse} of $\pi$, written $\rev(\pi)$, is formed by writing $\pi$ backwards: $\pi_n\pi_{n-1}\cdots \pi_1$.
\end{definition}

\begin{example}
    The flip of $836791245$ is $274319865$.  The reverse of the latter is $568913472$.
\end{example}

\begin{definition}
Given an SYT $T$, let $w$ be its reading word.  Apply the RSK insertion algorithm to $\rev(\flip(w))$.  We define $\phi(T)$ to be the resulting tableau.
\end{definition}

\begin{example}
Let $T$ be the following tableau: $$\young(8,3679,1245)$$ Its reading word is $836791245$.  Applying flip and reverse yields $568913472$, whose RSK insertion tableau is $\phi(T)$, shown below at left with its tableau of cocharge labels shown at right:
$$\young(5,3689,1247) \hspace{2cm} \young(2,1233,0012)$$
Notice that the descents of $T$ are $2,5,7$ and the major index is $2+5+7=14$.  The descents of $\phi(T)$ are $2,4,7$, and its cocharge is the sum of the cocharge labels, which is $14$.
\end{example}

\begin{lemma}
   The bijection $\phi$ is shape-preserving: $\sh(T)=\sh(\phi(T))$.
\end{lemma}
\begin{proof}
The shape of an RSK insertion is determined by the length of the longest increasing subsequence \cite{Fulton}, and then the longest pair of disjoint increasing subsequences, and so on.  These statistics are preserved by performing $\rev$ and $\flip$ in succession (each of which simply switch increasing to decreasing).      
\end{proof}

\begin{prop}
We have $\cc(\phi(T))=\maj(T)$ for any standard Young tableau $T$.    
\end{prop}

\begin{proof}
Let $w$ be the reverse of the flip of the reading word of $T$. Then $\cc(\phi(T))=\cc(w)$ since RSK insertion preserves cocharge.  

The cocharge subscripts on $w$ may be alternatively defined by labeling the $1$ with $0$, and then looking for $2,3,4,\ldots$ and whenever the next number is to the left of the previous, we increment the subscript, and if it is to the right, we do not increment, for instance:
$$5_2 6_2 8_3 9_3 1_0 3_1 4_1 7_2 2_0$$

This is the same as the \textit{charge} of the reversed word of $w$ (which is the flip of the reading word of $T$), in which we increment if we move to the right and do not otherwise:
$$2_0 7_2 4_1 3_1 1_0 9_3 8_3 6_2 5_2$$

The sum of the charge labels can be computed as follows: for all numbers after the first rightward move, we add $1$.  Then we add $1$ for all numbers after the next rightward move, and so on.  If we moved rightwards from $j$ to $j+1$, this means we added $n-(j-1)=n+1-j$ to the total charge for each such move.  The rightward moves in the flipped word correspond to the descents in the original word (in the running example, $836791245$), and those descents are precisely the values $n+1-j$ by the definition of the flip bijection.  So we have shown that $\cc(\rev(\flip(w))=\maj(w)$.  This completes the proof. 
\end{proof}

\begin{lemma}
    The map $\phi$ preserves the number of descents of the tableau.
\end{lemma}

\begin{proof}
    First, note that the number of descents is preserved under RSK because $i$ and $i+1$ cannot switch positions under any Knuth move in a permutation (see the definition of Knuth equivalence above, or from \cite{Fulton}).

Then, the descent pairs become ascents under $\flip$, which then become descents again under $\rev$.  This completes the proof.
\end{proof}

\subsection{Symmetric functions and graded Frobenius series}\label{sec:Frobenius}

The two bases of the ring of symmetric functions $$\Lambda_{\mathbb{C}}(x_1,x_2,\ldots)=\lim_{\leftarrow} \Lambda_{\mathbb{C}}(x_1,x_2,\ldots,x_n)=\lim_{\leftarrow}\mathbb{C}[x_1,\ldots,x_n]^{S_n}$$ that will be used in this paper are the elementary and Schur symmetric functions.  The elementary are defined as $e_d(x_1,x_2,\ldots)=\sum_{i_1<i_2<\cdots < i_d} x_{i_1}\cdots x_{i_d}$ and $e_\lambda=\prod e_{\lambda_i}$, and we can restrict these to $n$ variables as well to define the elementary symmetric polynomials.  The Schur functions can be defined by the combinatorial formula $s_\lambda=\sum_{T\in \SSYT(\lambda)}\prod_{c \in \lambda} x_{T(c)}$.

The Frobenius map encodes an $S_n$ representation as a Schur positive symmetric function.  In particular, $\Frob$ is the additive map on representations that sends the irreducible $S_n$-module $V_\nu$ to the Schur function $s_\nu$.  For instance, $\Frob(V_{(2)}\oplus V_{(1,1)}\oplus V_{(1,1)})=s_{(2)}+2s_{(1,1)}$.

For $S_n$-modules with a grading (such as the coinvariant ring, which is graded by degree), the \textbf{graded Frobenius map} is the generating function $$\grFrob_q(R):=\sum_d \Frob(R_d)q^d$$ where $R_d$ is the $d$-th graded piece and $q$ is a formal variable.

In this paper, we are working with doubly graded $S_n$-modules (by the $x$-degree and $y$-degree), and so we will be using the \textbf{bi-graded Frobenius map} $$\grFrob_{q,t}(R):=\sum_d \Frob(R_{d_1,d_2})q^{d_1}t^{d_2}.$$

The theorem of Lusztig and Stanley \cite{Stanley} mentioned above on the decomposition of $R_n$ into irreducibles can therefore be stated as follows:
$$\grFrob_q(R_n)=\sum_{T\in \SYT(n)} q^{\maj(T)}s_{\sh(T)}$$
which is also equal to $\sum_{T\in \SYT(n)} q^{\cc(T)}s_{\sh(T)}$ by the above discussion. The result of Lascoux and Sch\"utzenberger \cite{LascouxSchutzenberger} is that $$\grFrob_q(R_\mu)=\sum_{T\in \SSYT_\mu} q^{\cc(T)}s_{\sh(T)}$$ where $\SSYT_\mu$ is the set of SSYT's of content $\mu$.  

\subsection{Diagonal coinvariants and Garsia--Haiman modules}\label{sec:GH}

We work in the polynomial ring $\CC[{\bf x}, {\bf y}]=\CC[x_1,\ldots,x_n,y_1,\ldots,y_n]$ in two sets of variables.  The \textbf{diagonal action} of the symmetric group $S_n$ on $\CC[\bf x,\bf y]$ is given by permuting the variables: $$\pi f(x_1,\ldots,x_n,y_1,\ldots,y_n)=f(x_{\pi(1)},\ldots,x_{\pi(n)},y_{\pi(1)},\ldots, y_{\pi(n)}).$$

\begin{definition}
  The ring of \textbf{diagonal coinvariants} is given by $$\DR_n=\CC[{\bf x}, {\bf y}]/I_n$$ where $I_n$ is the ideal generated by the positive-degree $S_n$-invariants under the diagonal action.
\end{definition}

\begin{remark}
We use $\DR_n$ rather than $R_n$ here to distinguish the diagonal case from the one-variable coinvariant ring $R_n=\CC[x_1,\ldots,x_n]/(e_1,\ldots,e_n)$.  We similarly use $\DR_\mu$ rather than $R_\mu$ below for the Garsia--Haiman modules to distinguish them from the Garsia--Procesi modules $R_\mu$.  
\end{remark}

The \textbf{Garsia--Haiman modules} are quotients of $\DR_n$ defined as follows.  We use the differential inner product on $\CC[\mathbf{x},\mathbf{y}]$ that defines $\langle f, g \rangle$ to be the constant term of $\partial f(g)$, where $\partial f$ is the polynomial differential operator formed by changing all $x_i$'s and $y_i$'s to $\frac{\partial}{\partial x_i}$ and $\frac{\partial}{\partial y_i}$ respectively.

\begin{definition}
 For a partition $\mu$ of $n$, let the squares in its Young diagram have coordinates $(a_1,b_1),\ldots,(a_n,b_n)$ in some order, where we start with the lower leftmost square having coordinates $(1,1)$. 
 
 Define $\Delta_\mu=\det M_\mu$ where $M_\mu$ is the matrix whose $(i,j)$ entry is $x_i^{a_j}y_i^{b_j}$.
\end{definition}

\begin{definition}
 Define $I_\mu$ to be the ideal of all polynomials $f\in \mathbb{C}[{\bf x},{\bf y}]$ such that $\partial f(\Delta_\mu)=0$.  The \textbf{Garsia--Haiman module} with index $\mu$ is $$\DR_\mu =\CC[{\bf x},{\bf y}]/I_\mu.$$
\end{definition}

For the case when $\mu$ is a \textbf{hook shape} of the form $(n-k+1,1^{k-1})$, several monomial bases were established in \cite{HookBases}.  The graded Frobenius was derived by Stembridge \cite{Stembridge}, and makes use of the following statistics.

\begin{definition}\label{def:maj-comaj}
    For a standard tableau $T$ of size $n$, define $\maj_{i,j}(T)$ to be the sum of the descents $d$ of $T$ such that $i\le d<j$.

    Define $\comaj_{i,j}(T)$ to be the sum of the values $n-d$ over all descents $d$ such that $i\le d <j$.
\end{definition}

Then for $\mu=(n-k+1,1^{k-1})$, we have \cite{Stembridge} 
\begin{equation}\label{eq:stembridge}\grFrob_{q,t}(\DR_\mu)=\sum_{T\in \SYT(n)} q^{\maj_{1,n-k+1}(T)}t^{\comaj_{n-k+1,n}(T)}s_{\sh(T)}.\end{equation}

\subsection{Young symmetrizers and one-variable higher Specht polynomials}

We now recall the construction of the one-variable higher Specht basis for the coinvariant ring $R_n$.  Working in the group algebra $\QQ[S_n]$, define the \textbf{Young idempotent} $\varepsilon_T \in \QQ[S_n]$ by $$\varepsilon_T=\sum_{\tau \in C(T)}\sum_{\sigma \in R(T)}\sgn(\tau) \tau \sigma$$ where $C(T)\subseteq S_n$ is the group of \textit{column permutations}  of $T$ (those that send every number to another number in its column in $T$), and $R(T)\subseteq S_n$ is the group generated by row permutations.

Define the Young (anti)symmetrizers $\alpha$ and $\beta$ as follows.  For any subgroup $U\subseteq S_n$, define $$\alpha(U)=\sum_{\tau \in U}\sgn(\tau)\tau\hspace{0.5cm}\text{ and }  \hspace{0.5cm}\beta(U)=\sum_{\sigma\in U}\sigma.$$  In this notation, the Young idempotent $\varepsilon_T$ can be written as $$\varepsilon_T=\alpha(C(T))\beta(R(T)).$$

Notice that the ordinary Specht polynomials $F_T$ defined in the introduction may alternatively be defined as follows.  Let $$x_T=\prod_{i=1}^n x_i^{r_i-1}$$ where $r_i$ is in the row in which the letter $i$ appears in $T$ (with the bottom row being row $1$).  Then $$\varepsilon_T x_T=\sum_{\sigma\in R(T)}\sum_{\tau\in C(T)}\sgn(\tau)\tau\sigma x_T=|R(T)|\sum_{\tau\in C(T)}\sgn(\tau)=|R(T)|F_T,$$ so the polynomials $\varepsilon_T x_T$ are scalar multiples of the $F_T$ basis and hence form a basis themselves.  

As mentioned in the introduction, the one-variable higher Specht polynomials constructed in \cite{ATY} are constructed in a similar manner.  Let $(T,S)$ be a pair of standard Young tableau of the same shape $\lambda$, and let $\ccind(S)$ be the tableau consisting of the cocharge subscripts of the reading word of $S$, written in the corresponding squares of the diagram of $\lambda$.  Define $$x_T^S=\prod_c x_{T(c)}^{\ccind(S)(c)}$$ where the product is over all squares $c$ in the diagram of $\lambda$.   Then $F_T^S$ is defined as $\varepsilon_T x_T^S$, and the set of polynomials $F_T^S$ is the higher Specht basis for $R_n$.

\section{General constructions in two sets of variables}\label{sec:general}

We begin by generalizing some of the basic theory of higher Specht polynomials (as established in \cite{AllenOneVariable}, \cite{ATY}, and \cite{GillespieRhoades}) to the multivariate setting.  In this section we are working in the polynomial ring $\mathbb{C}[\mathbf{x},\mathbf{y}]$.  

For a tableau $T\in \Tab(\lambda)$, fix an ordering on the squares of $\lambda$ and write $T_1,T_2,\ldots,T_n$ to denote the values of $T$ in each of those squares.  Also, for any tuples $c=(c_1,\ldots,c_n)$ and $d=(d_1,\ldots,d_n)$ of nonnegative integers, we write $$x_T^c=x_{T_1}^{c_1} x_{T_2}^{c_2}\cdots x_{T_n}^{c_n}\qquad\text{ and }\qquad y_T^d=y_{T_1}^{d_1} y_{T_2}^{d_2}\cdots y_{T_n}^{d_n}.$$

\begin{definition}
We write $$F_T^{c,d}=\varepsilon_T (x_T^cy_T^d).$$
\end{definition}

\begin{remark}
    All proofs in this section  hold for any number of variables; for instance, we could have three sets of variables $\mathbf{x},\mathbf{y},\mathbf{z}$ and three exponent sequences $c,d,e$ and define $F_T^{c,d,e}=\varepsilon_T (\mathbf{x}^c\mathbf{y}^d\mathbf{z}^e)$ and we would obtain analogous results.
\end{remark}

We will first show the general statement that, assuming the polynomials $F_T^{c,d}$ are independent for $T$ standard of shape $\lambda$, the submodule $V^{c,d}\subseteq \mathbb{C}[\mathbf{x},\mathbf{y}]$ spanned by these polynomials is a copy of the irreducible $S_n$-module $V_\lambda$.  We first recall (see, for instance, \cite{Peel}) the Garnir relations that govern the $S_n$-module structure of $V_\lambda$ with respect to the standard Specht basis.

\begin{definition}
  Let $T\in \Tab(\lambda)$.  Let $a$ and $b$, with $a<b$, be the indices of two distinct columns of $T$, and let $t\le \lambda'_b$ be a row index of one of the entries of column $b$.  Let $S^{a,b}_t$ be the subgroup of $S_n$ consisting of all permutations of the set of elements of $T$ residing either in column $a$ weakly above $t$, or in column $b$ weakly below $t$. 
  
  The \textbf{Garnir element} $G^{a,b}_t$ is the partial antisymmetrizer $$G^{a,b}_t:=\sum_{\omega\in S^{a,b}_t} \mathrm{sgn}(\omega)\omega.$$
\end{definition}

The following lemma states that Lemma 3.3 in \cite{Peel} (or equivalently Lemma 3.13 in \cite{GillespieRhoades}) generalizes to this setting.

\begin{lemma}\label{lem:alpha}
  Let $U$ be any subgroup of $S_n$ and let $C=C(T)$ where $T\in \Tab(\lambda)$.  Suppose there is an involution $\sigma\mapsto \sigma'$ on $UC$ such that for each $\sigma\in UC$, there exists $\rho_\sigma\in R(T)$ for which $\rho_\sigma^2=1$, $\sgn(\rho_\sigma)=-1$, and $\sigma'=\sigma\rho_\sigma$.  Then $$\alpha(U) F_T^{c,d}=0.$$
\end{lemma}

For a proof of this lemma, we simply refer to the proof of Lemma 3.13 in \cite{GillespieRhoades}, and note that the proof only depended on the fact that the polynomials $F_T^S$ were defined as $\varepsilon_T$ applied to a monomial; the specific definition of the monomial did not matter.

\begin{prop}\label{prop:Garnir}
We have $\pi F_T^{c,d} = F_{\pi T}^{c,d}$ for any $\pi\in S_n$, and the $F_T^{c,d}$ elements satisfy the Garnir relations  $G^{a,b}_{t}(F_T^{c,d})=0$.
\end{prop}

\begin{proof}
    We have 
    \begin{align*}
        \pi F_T^{c,d} &= \pi \sum_{\tau \in C(T)}\sum_{\sigma \in R(T)}\sgn(\tau) \tau \sigma x_T^cy_T^d\\
         &= \sum_{\tau\in C(T)}\sum_{\sigma\in R(T)}\sgn(\tau)\pi\tau(\pi^{-1}\pi)\sigma\pi^{-1}x_{\pi(T)}^c y_{\pi(T)}^d \\
         &= \sum_{\tau\in C(T)}\sum_{\sigma\in R(T)}\sgn(\tau)(\pi\tau\pi^{-1})(\pi\sigma\pi^{-1})x_{\pi(T)}^c y_{\pi(T)}^d \\
         &= \sum_{\tau\in \pi C(T)\pi^{-1}}\sum_{\sigma\in \pi R(T)\pi^{-1}}\sgn(\tau)\tau \sigma x_{\pi(T)}^c y_{\pi(T)}^d \\
         &= \sum_{\tau\in C(\pi T)}\sum_{\sigma\in R(\pi T)}\sgn(\tau)\tau \sigma x_{\pi(T)}^c y_{\pi(T)}^d \\
         &= F^{c,d}_{\pi(T)}.
    \end{align*}  

    For the Garnir relations, the proof now exactly follows that of Theorem 3.1 in \cite{Peel} for the ordinary Specht polynomials starting from Lemma \ref{lem:alpha}, since the remaining steps of Peel's proof only depend on the operators $\alpha(U)$ and not on the specific polynomials they are applied to.  We therefore omit the rest of the details and refer the reader to \cite{Peel}.
\end{proof}

\begin{prop}\label{prop:irreducible}
 The subspace $M_\lambda=\mathrm{sp}(F_T^{c,d}:T\in \Tab(\lambda))\subseteq \mathbb{C}[\mathbf{x},\mathbf{y}]$ is generated by the subset $\mathbb{B}=\{F_T^{c,d}:T\in \SYT(\lambda)\}$.  Moreover, if the polynomials in $\mathbb{B}$ are independent, then $M_\lambda\cong V_\lambda$ as an $S_n$-module, and there is an action-preserving isomorphism induced by the bijection $F_T^{c,d}\to F_T$ that sends the basis $\mathbb{B}$ to the ordinary Specht basis.
\end{prop}

\begin{proof}
This follows from the well-known Garnir straightening algorithm for polynomials that satisfy the Garnir relations (see \cite{Sagan}) along with Proposition \ref{prop:Garnir}.
\end{proof}

\subsection{Conditions for independence}

Proposition \ref{prop:irreducible} shows that if the basis elements are independent, then $M_\lambda$ is a copy of the irreducible Specht module.  We now show some conditions for determining independence in quotients by $S_n$-invariant ideals or in the full polynomial ring. 

\begin{prop}\label{prop:nonzero}
    The polynomials in $\mathbb{B}=\{F_T^{c,d}:T\in \SYT(\lambda)\}$ are independent in $\mathbb{C}[\mathbf{x},\mathbf{y}]/I$ if and only if some particular element $F_T^{c,d}$ is nonzero in the quotient (i.e., not in the ideal $I$).  
\end{prop}

\begin{proof}
    Let $M_\lambda$ be the span of $\mathbb{B}$.  Then by the discussion above, there is a natural map of $S_n$-modules $V_\lambda\to M_\lambda$ given by sending the ordinary Specht module generator for the tableau $T$ to $F_T^{c,d}$.  Consider the decomposition of $M_\lambda$ into irreducible representations; by composing with the projection onto any factor, we have by Schur's lemma that the map can only be nonzero on factors isomorphic to $V_\lambda$.  Since we are assuming there is a nonzero element in $M_\lambda$ mapped to by a nonzero element of $V_\lambda$, it follows that the decomposition of $M_\lambda$ into irreducibles contains at least one copy of $V_\lambda$.

    Since $|\mathcal{B}|=\dim(V_\lambda)$ we have $\dim(M_\lambda)\le \dim(V_\lambda)$, so by the dimension inequalty we have $M_\lambda\cong V_\lambda$ as desired.
\end{proof}

We now show a sufficient condition for one of the $F_T^{c,d}$ elements to be nonzero in the full polynomial ring $\mathbb{C}[\mathbf{x},\mathbf{y}]$ that we will use below.

\begin{prop}\label{prop:SSYT-values}
    Suppose that $c=(c_1,\ldots,c_n)$ and $d=(d_1,\ldots,d_n)$ are tuples of nonnegative integers, with their ordering corresponding to a chosen ordering of the boxes of the Young diagram of a partition $\lambda$ of $n$.  Suppose further that there exists a total ordering on the possible pairs of values $(m,n)$ that arise among the pairs $c_i,d_i$ such that, when the boxes of $\lambda$ are filled in order with $(c_i,d_i)$, it forms a semistandard Young tableau with respect to the ordering on the pairs.

    Then for any standard tableau $T$ of shape $\lambda$, the polynomial $F_{T}^{c,d}$ is nonzero in $\mathbb{C}[x,y]$.  In particular, the leading term $x_T^cy_T^d$ does not vanish.
\end{prop}

\begin{example}\label{ex:horz-strips}
    Suppose we choose the ordering on pairs $(m,n)$ by $(m,n)>(x,y)$ if and only if either $m>x$ or $m=x$ and $n>y$.  Consider the following tableau, where pairs $(m,n)$ are written as $mn$:
    $$\begin{ytableau}
        22 & 22 \\
        21 & 21 & 22 & 22 \\
        01 & 10 & 10 & 13 \\
        00 & 00 & 00 & 01 & 10
    \end{ytableau}$$
    Also, for simplicity let $T$ be the tableau that standardizes the above (the choice of $T$ will not matter in the proof): $$\begin{ytableau}
        12 & 13 \\ 
        10 & 11 & 14 & 15 \\
        4 &  6 & 7  & 9 \\
        1 & 2 & 3 & 5 & 8
    \end{ytableau}$$
    Then $x_T^c y_T^d=x_{15}^2 y_{15}^2 \cdot x_{14}^2y_{14}^2\cdot x_{13}^2y_{13}^2\cdot x_{12}^2y_{12}^2\cdot x_{11}^2y_{11}\cdots $.  In the proof below, we will show that this leading term does not cancel after applying $\varepsilon_T$.
\end{example}

\begin{proof}
    We will show that any term $\tau\sigma$ arising in $\varepsilon_T$ that fixes the leading term $x_T^c y_T^d$ actually fully lies in $R(T)$, so that $\tau$ must be the identity permutation and the sign is positive.  This shows that the leading term has a nonzero coefficient, so the polynomial is nonzero.

    We first make some notational simplfications; the proof will not depend on the choice of $T$, so we may assume $T$ is the standardization of the tableau labeled by the $(c_i,d_i)$ pairs and that this standardization order is the ordering on the boxes in $\lambda$ as well.  This means that $(c_1,d_1)\le (c_2,d_2)\le \cdots \le (c_n,d_n)$ with respect to the total ordering on pairs, and for the purposes of showing that a coefficient is nonzero, we can change the pairs $(c_i,d_i)$ to numbers $b_1\le \ldots\le b_n$ such that $b_i=b_{i+1}$ if and only if $(c_i,d_i)=(c_{i+1},d_{i+1})$.  We can then replace the monomial $x_T^cy_T^d$ with $x_T^b$ and reduce to the standard action of $S_n$ rather than the diagonal action.  
    
    For instance, in Example \ref{ex:horz-strips}, the tableau of pairs can be replaced with the SSYT below:
    $$\begin{ytableau}
        6 & 6 \\
        5 & 5 & 6 & 6 \\
        2 & 3 & 3 & 4 \\
        1 & 1 & 1 & 2 & 3 
    \end{ytableau}$$
    and the monomial $x_T^b$ is $$x_{15}^6x_{14}^6x_{13}^6x_{12}^6x_{11}^5x_{10}^5x_{9}^4x_{8}^3x_7^3x_6^3x_5^2x_4^2x_3x_2x_1.$$
    We first show that if $\tau\in C(T)$ and $\sigma \in R(T)$ and $\tau\sigma(x_T^b)=x_T^b$, then if $n-r,n-r+1,\ldots,n$ are the entries in the row of $n$ such that $b_{n-r}=b_{n-r+1}=\cdots = b_n$, then  $\tau\sigma(n)$ is among $n-r,\ldots,n$.  (In the example above, $b_{14}=b_{15}=6$ are the entries in the row of $15$ with the largest exponent, and we want to show that $\tau\sigma(15)$ is either $14$ or $15$.) 

    Indeed, in the example, since it fixes $x_T^b$, $\tau\sigma$ must send $15$ to either $12, 13, 14$, or $15$ so that it retains its exponent.  But $\sigma$ cannot send any element in a row below  the $15$, like $9$ or $5$, to $12, 13, 14,$ or $15$, so then $\tau$ cannot send $15$ to an element with exponent $6$ without $\sigma$ already sending it there.  In general, this reasoning means $\tau \sigma$ must permute the entries $n-r,\ldots,n$.
    By a similar argument, $\tau\sigma$ must permute the entries in the next row up that also have value $b_n$, and so on, so all of the entries for which $b_i=b_n$ are sent to other entries in their row with the same $b$ value. 

    Since these entries are fixed amongst themselves, we then consider the second highest value of $b_i$ (in this example, $5$), and repeat the argument, and so on.  Thus $\tau\sigma\in R(T)$ as desired.
\end{proof}

\section{A higher Specht basis for the hook-shape Garsia--Haiman modules}\label{sec:hooks}

Throughout this subsection, we set $\mu=(n-k+1,1^{k-1})$ to be the \textit{hook shape} of height $k$ and size $n$.  We now define our higher Specht polynomials for the hook shape Garsia--Haiman modules.

\begin{definition}
  For $S\in \SYT(n)$, define the \textbf{$\mu$-cocharge tableau} of $S$, written $\ccind_\mu(S)$, to be the tableau of shape $\sh(S)$ that has $0$ in the squares occupied by $1,2,\ldots,n-k+1$ in $S$, and where for $n-k+1,n-k+2,\ldots,n$ we apply the cocharge algorithm, incrementing the label whenever we do not wrap around in reverse reading order.
  
 The \textbf{$\mu$-cocharge} of $S$, written $\cc_\mu(S)$, is the sum of the entries of $\ccind_\mu(S)$.
\end{definition}

\begin{definition}
    For $S\in \SYT(n)$, define the \textbf{reverse $\mu$-cocharge tableau}, written $\ccind'_\mu(S)$, to be the tableau of shape $\sh(S)$ that has $0$ in the squares occupied by $n-k+1,\ldots,n$ in $S$, and where for $n-k+1,n-k,n-k-1,\ldots,1$ we calculate cocharge in \textit{reverse} (forward reading order and from biggest number to smallest), incrementing when we don't wrap around.  
    
    The \textbf{$\mu$-reverse cocharge}, written $\cc'_\mu(S)$, is the sum of the entries of the reverse $\ccind'_\mu(S)$.
\end{definition}

\begin{example}\label{ex:cctab}
  Suppose $n=7$, $S$ is the tableau $$\young(67,35,124),$$ and $k=4$.  Then $\ccind_\mu(S)$ is the tableau $$\young(22,01,000),$$ and $\ccind'_\mu(S)$ is the tableau $$\young(00,00,110).$$
\end{example}

\begin{definition}
  Let $(T,S)\in \mathrm{Tab}(\lambda)\times \SYT(\lambda)$ for a fixed shape $\lambda$. The $\mu$-\textbf{monomial} of $(T,S)$ is $$\mathbf{xy}_T^S:=\prod_{b\in \lambda} x_{T(b)}^{\ccind_\mu(S)(b)} y_{T(b)}^{\ccind'_\mu(S)(b)}.$$  Thus we have $\cc_\mu'(S)=1+1=2$ and $\cc_\mu(S)=1+2+2=5$. Notice also that $n-k+1=4$ and $\maj_{1,4}(S)=2$ and $\comaj_{4,7}(S)=(7-4)+(7-5)=5$.
\end{definition}

\begin{lemma}\label{lem:degrees}
    We have that $\cc_\mu(S)=\comaj_{n-k+1,n}(S)$ and $\cc_\mu'(S)=\maj_{1,n-k+1}(S)$ for any standard Young tableau $S$ (see Definition \ref{def:maj-comaj}).
\end{lemma}

\begin{proof}
    To compute $\cc_\mu(S)$, notice that we begin at $n-k+1$ and label with a subscript of $0$, and then continue to find $n-k+2,n-k+3,\ldots,n$.  If we are at a descent, we increment the subscript label at the next step.  Thus each descent $d$ we encounter contributes $n-d$ to the total cocharge, and so the total is equal to $\comaj_{n-k+1,n}(S)=\sum(n-d)$ where the sum is over descents $d$ such that $n-k+1\le d<n$.

    For $\cc'_\mu(S)$, we use the reverse and flip maps on the reading word established in Section \ref{sec:GH}, and then by the same argument as above, we see that $\cc'_\mu(S)=\maj_{1,n-k+1}(S)$.
\end{proof}

\begin{definition}
  For each pair $(T,S)$ as above, we define the polynomial $$F_{T}^S=\varepsilon_T \mathbf{xy}_T^S.$$
\end{definition}

\begin{example}
    Keeping $S$ as in Example \ref{ex:cctab} and $k=4$, we set $$T=\young(56,24,137)\hspace{1cm}\text{and} \hspace{1cm}S=\young(67,35,124)$$ and we find $$\mathbf{xy}^S_T=x_{4}x_5^2x_6^2y_1y_3.$$
\end{example}

\begin{example}
    Setting $$T=\young(34,12), \hspace{1cm}S=\young(24,13)$$ and $k=2$, we have $\mu=(3,1)$ and so $$\ccind_\mu(S)=\young(01,00),\hspace{1cm}\ccind_\mu'(S)=\young(00,10).$$  We then have  $\mathbf{xy}^S_T=x_4y_1$ and $$F_T^S= x_4y_1 + x_3y_1 + x_4y_2 + x_3y_2 - 2 (x_2y_1 + x_1y_2 + x_4y_3 + x_3y_4)+ x_1y_4 + x_1y_3 + x_2y_4 + x_2y_3.$$
\end{example}

We first use our general theory from Section \ref{sec:general} show that for any fixed $S$ of shape $\lambda$, the polynomials $F_T^S$ are indeed higher Specht polynomials in the full polynomial ring $\mathbb{C}[\mathbf{x},\mathbf{y}]$.

\begin{prop}\label{prop:if-basis-then-specht}
The polynomials $F_T^S$ for $T\in \SYT(\lambda)$ span a copy of $V_\lambda$ in $\mathbb{C}[\mathbf{x},\mathbf{y}]$.
\end{prop}

\begin{proof}
We show that the values of $c:=\ccind_\mu(S)$ and $d:=\ccind'_\mu(S)$ form a semistandard Young tableau according to a particular ordering on the pairs of values, so that it satisfies the conditions of Proposition \ref{prop:SSYT-values}.  Then, the result will follow by Proposition \ref{prop:irreducible}.

First, to define the ordering, note that the nonzero entries of the $\mu$-cocharge tableau are disjoint from the nonzero entries of the reverse $\mu$-cocharge tableau.  We then negate the values of the reverse $\mu$-cocharge tableau and define the ordering on the pairs of values in each box by their sum.  (See Figure \ref{fig:cctab-ssyt}.)

Now, consider the $\mu$-cocharge labeling.  Starting at the square containing $n-k+1$, we label it with $0$, and then label $n-k+2,n-k+3,\ldots,n$ with subscripts that increment when not wrapping around the reading word in reverse reading order.  By standardness, the squares labeled $i$ in this process for any given $i$ form a horizontal strip, and so the labels form a semistandard skew tableau.  

The same holds for the labels of the reverse $\mu$-cocharge tableau, and the squares labeled $0$ in both tableaux are precisely the horizontal strip consisting of the longest sequence of consecutive numbers including $n-k+2$ that occurs in order in the reading word.    This completes the proof.
\end{proof}

\begin{figure}
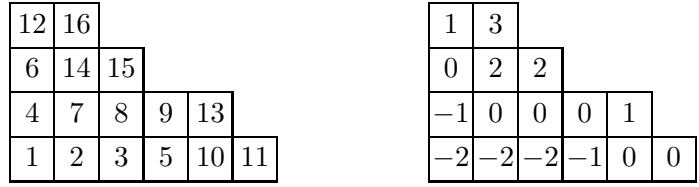

    \centering
    $$\begin{ytableau}
        12 & 16 \\
        6 & 14 & 15 \\
        4 & 7 & 8 & 9 & 13 \\
        1 & 2 & 3 & 5 & 10 & 11
    \end{ytableau}\hspace{2cm}
    \begin{ytableau}
        1 & 3 \\
        0 & 2 & 2 \\
         -1 & 0 & 0 & 0 & 1 \\
        -2 & -2 & -2 & -1 & 0 & 0
    \end{ytableau}$$
    \caption{A tableau $S$ at left, and its values of $\ccind_\mu(S)$ and the negated labels from $\ccind'_\mu(S)$ superimposed at right.  Here $n=16$ and $k=9$, so $n-k+1=8$.}
    \label{fig:cctab-ssyt}
\end{figure}

We now note that we can recover the one-variable higher Specht polynomials from our construction by a substitution and a degree shift.

\begin{lemma}\label{lem:x-inverse}
    Consider the map to the Laurent polynomial ring (also defined in \cite{HookBases}) $$\psi:\mathbb{C}[x_1,\ldots,x_n,y_1,\ldots,y_n]\to \mathbb{C}[x_1,\ldots,x_n,x_1^{-1},\ldots,x_n^{-1}]$$ that sends $y_i\mapsto x_i^{-1}$.   Then if $q$ is the largest entry in $\ccind'_\mu(S)$, we have $$\psi(F_T^S(\mathbf{x},\mathbf{y}))\cdot x_1^qx_2^q\cdots x_n^q=F_T^S(\mathbf{x})$$ where $F_T^S(\mathbf{x})$ is the one-variable higher Specht polynomial  defined by Ariki, Terasoma, and Yamada (see Equation \eqref{eq:ATY}).
\end{lemma}

\begin{proof}
Applying the map $\psi$ effectively interprets the cocharge values for $y$ in $\ccind'_\mu(S)$ as negative entries.  Shifting all of these values up by $q$ gives the same descent set and therefore the same cocharge values as computing ordinary cocharge of $S$. The result follows.  
\end{proof}

\subsection{Independence in the full polynomial ring}

We have shown that for any fixed $S$, the set of polynomials $F_{T}^S$ ranging over all $T$ of the same shape as $\lambda$ generates a copy of $V_{\sh(S)}$.  We now show that these copies are all distinct for different $S$.  For notational brevity, we make the following definition.

\begin{definition}
    We define $M_\lambda^S$ to be the copy of $V_\lambda$ inside $\mathbb{C}[\mathbf{x},\mathbf{y}]$ generated by the polynomials $F_T^S$ over all $T\in\SYT(\sh(S))$.
\end{definition}

\begin{lemma}\label{lem:disjoint}
    If $S_1$ and $S_2$ are two standard Young tableaux of size $n$ and shapes $\lambda$ and $\rho$ respectively, then $M_\lambda^{S_1}$ and $M_{\rho}^{S_2}$ are disjoint submodules of $\mathbb{C}[\mathbf{x},\mathbf{y}]$.
\end{lemma}

\begin{proof}
    If $\lambda\neq \rho$, this is clear by the uniqueness of the decomposition into irreducible $S_n$ representations.

    If $\lambda=\rho$, let the multisets of labels in $\ccind_\mu(S_1)$ and $\ccind_\mu(S_2)$ respectively be denoted $C_{S_1}$ and $C_{S_2}$, and let the multisets of labels in $\ccind_\mu'(S_1)$ and $\ccind_\mu(S_2)'$ respectively be denoted $C_{S_1}'$ and $C_{S_2}'$.  If either $C_{S_1} \neq C_{S_2}$ or $C'_{S_1} \neq C'_{S_2}$, then the exponent tuples on the $x$ and $y$ variables involved in each term of any element of $M_\lambda^{S_1}$ are constant and differ from those of $M_\lambda^{S_2}$, and so the two cannot be the same submodule.

    Finally, if $\lambda=\rho$ and $C_{S_1}=C_{S_2}$ and $C'_{S_1}=C'_{S_2}$, then in particular the largest degree, $d$, in the $y$ variables for the term of every element in each module is equal.  We consider the map $\psi$ as in Lemma \ref{lem:x-inverse}, and consider the images of $M_\lambda^{S_1}$ and $M_{\lambda}^{S_2}$.  Certainly $\psi$ preserves the $S_n$ action. If we multiply every element in $M_\lambda^{S_1}$ and $M_{\lambda}^{S_2}$ by $x_1^d\cdots x_n^d$, by Lemma \ref{lem:x-inverse}, we obtain modules $\widetilde{M}_{\lambda}^{S_1}$ and $\widetilde{M}_{\lambda}^{S_2}$ generated by precisely the one-variable higher Specht polynomials $F_T^S(\mathbf{x})$.  These modules are disjoint \cite{ATY}, and therefore $M_\lambda^{S_1}$ and $M_\lambda^{S_2}$ are disjoint as well.
\end{proof}

\subsection{Independence in \texorpdfstring{$\DR_\mu$}{}}

We now recall the a generating set for the ideal $I_\mu$ defining $\DR_\mu$ in the hook shape case, and a sub-ideal defined in \cite{HookBases}.

\begin{prop}[\cite{HookBases}] \label{prop:ideal}
    If $\mu=(n-k+1,1^{k-1})$, the ideal $I_\mu$ is generated by:
    \begin{enumerate}
        \item The elementary symmetric functions $e_1(\mathbf{x}),\ldots,e_n(\mathbf{x})$,
        \item The elementary symmetric functions 
        $e_1(\mathbf{y}),\ldots,e_n(\mathbf{y})$,
        \item All products $x_{i_1}\cdots x_{i_k}$ of $k$ distinct $x$ variables,
        \item All products $y_{j_1}\cdots y_{j_{n-k+1}}$ of $n-k+1$ distinct $y$ variables.
        \item The products $x_iy_i$ for $i=1,\ldots,n$.
    \end{enumerate}
\end{prop}

\begin{definition}[\cite{HookBases}]
    Define $\mathcal{I}_k$ to be the ideal generated by  the last three bullet points in the definition above, that is, just the products $x_{i_1}\cdots x_{i_k}$, $y_{j_1}\cdots y_{j_{n-k+1}}$, and $x_iy_i$.  Define $$\mathcal{P}_{n}^{(k)}=\mathbb{C}[\mathbf{x},\mathbf{y}]/\mathcal{I}_k.$$
\end{definition}

\begin{lemma}
    The elements $F_T^S$ are all nonzero in $\mathcal{P}_n^{(k)}$.
\end{lemma}

\begin{proof}
    We can think of the image of a polynomial in $\mathcal{P}_n^{(k)}$ as setting all terms that contain an $x_iy_i$ or a product of $k$ different $x$ variables or $n-k+1$ different $y$ variables to $0$.  If any remain, the image is nonzero in the quotient. (See \cite{HookBases}, Section 4, for a similar discussion).  In $F_T^S$, every term by construction has no $x_iy_i$ products and no more than $k-1$ of them for $x$ and no more than $n-k$ for $y$.  So in fact none of its terms vanish, so it is nonzero in $\mathcal{P}_n^{(k)}$.
\end{proof}

To show they descend to a basis of $\DR_\mu$, we recall the definition of the $S_n$-invariant basis $e_\nu^{(k)}$ of $(\mathcal{P}_n^{(k)})^{S_n}$ from \cite{HookBases}.

\begin{definition}
    For $d=1,\ldots,n$, define $$e_d^{(k)}=\begin{cases}
        e_d(x_1,\ldots,x_n) & d\le k-1 \\
        e_{n-d}(y_1,\ldots,y_n) & d\ge k
    \end{cases}$$ and $e_\nu^{(k)}=\prod_i e_{\nu_i}^{(k)}$.
\end{definition}

We now show that $\{F_T^S\cdot e_\nu^{(k)}\}$, ranging over all pairs $(T,S)$ of SYT's of the same shape and all $\nu$ such that the longest part is no more than $n$, forms a basis of $\mathcal{P}_n^{(k)}$.  We begin by analyzing these polynomials in the full polynomial ring.

\begin{lemma}\label{lem:MS}
    For a fixed $S$ (with say shape $\lambda$) and fixed shape $\nu$ with $\nu_1\le n$, the set of polynomials $$\{F_T^S\cdot e_\nu^{(k)}:T\in \SYT(\lambda)\}$$ span a copy of $V_\lambda$ in $\mathbb{C}[\mathbf{x},\mathbf{y}]$, which we call $M^{(k)}_{\lambda}(S,\nu)$.  Moreover, they form a higher Specht basis for this subspace.
\end{lemma}

\begin{proof}
    The polynomials $e_\nu^{(k)}$ are invariant under the diagonal $S_n$ action, and therefore $S_n$ acts identically on the polynomials $F_T^S\cdot e_\nu^{(k)}$ as they on $F_T^S$.  They are still nonzero polynomials since both $F_T^S$ and $e_\nu^{(k)}$ are nonezero.  This completes the proof.
\end{proof}

\begin{lemma}
    The modules $M^{(k)}_\lambda(S,\nu)$, for $S\in \SYT(n)$ and $\nu_1\le n$, are independent of each other in $\mathbb{C}[\mathbf{x},\mathbf{y}]$.
\end{lemma}

\begin{proof}
  For any fixed $\nu$, all the $M_\lambda^{(k)}(S,\nu)$ bases are independent of each other because we are simply multiplying the independent modules $M_\lambda^S$ through by the same $e_\nu^{(d)}$.  
    
   For differing $\nu$'s, say $\nu\neq \rho$, all elements of $M_\lambda^{(k)}(S,\nu)$ are divisible by $e_\nu^{(k)}$ and all elements of $M_\lambda^{(k)}(S,\rho)$ are divisible by $e_{\rho}^{(k)}$.  Thus, to show they are independent, we simply must show that some element of one is not divisible by the other's elementary symmetric function.  In particular, if it is, then we see that some $F_T^S$ is divisible by some $e_d^{(k)}$ in the full polynomial ring, since the two products of elementary symmetric functions differ and we are working in a unique factorization domain (the full polynomial ring $\mathbb{C}[\mathbf{x},\mathbf{y}]$).

   Suppose for contradiction that we have \begin{equation}\label{eq:g}
       F_T^{s}(\mathbf{x},\mathbf{y})=g(\mathbf{x},\mathbf{y})e_d^{(k)}(\mathbf{x})
   \end{equation} (where we are assuming without loss of generality that $d\le k-1$, so that $e_d^{(k)}$ is only in the $x$ variables.  A similar argument will work for $e_d^{(k)}(\mathbf{y})$ when $d\ge k$).  Then by Lemma \ref{lem:x-inverse}, setting $y_i=x_i^{-1}$ and multiplying through by each $x_i^q$, we have $$F_T^S(\mathbf{x})=x_1^q\cdots x_n^q \cdot g(\mathbf{x},\mathbf{x}^{-1})e_d(\mathbf{x}),$$ but the highest power of any $y$ variable in $g(x,y)$ must match that of $F_T^S(\mathbf{x},\mathbf{y})$ by Equation \eqref{eq:g}, so the $x_i^q$ powers cancel all the negative $x$ powers in $g(x,x^{-1})$ as well.  Thus $F_T^S$ is divisible by $e_d$ in the ordinary one-variable polynomial ring.  This is a contradiction, since it was proven in \cite{ATY} that the higher Specht basis for $R_n$ does not vanish modulo the ideal $(e_1,\ldots,e_n)$.  
\end{proof}

\begin{lemma}\label{lem:still-nonzero}
    The modules $M^{(k)}_\lambda(S,\nu)$, for $S\in \SYT(n)$ and $\nu_1\le n$, are also nonzero in the quotient $\mathcal{P}_n^{(k)}$.
\end{lemma}

\begin{proof}
   By Proposition \ref{prop:nonzero}, it suffices to show that $F_T^S\cdot e_\nu^{(k)}$ survives for each $S$ and $\nu$.   We show that $F_S^S\cdot e_\nu^{(k)}$ is nonzero.

   By the definition of the ideal defining $\mathcal{P}_n^{(k)}$, it suffices to show that some term in the product that has a nonzero coefficient is not divisible by any product $x_iy_i$, nor by $k$ distinct $x$ variables or $n-k+1$ distinct $y$ variables.  We note that, by Proposition \ref{prop:SSYT-values}, the leading monomial $\mathbf{xy}_S^S$ appears with a nonzero coefficient in $F_S^S$.  Because we are taking $S$ to also be the tableau $T$, this leading monomial is of the form $$x_n^{d_n}x_{n-1}^{d_{n-1}}\cdots x_{n-k+2}^{d_{n-k+2}} y_{n-k}^{d_{n-k}}\cdots y_1^{d_1}$$ where $$d_n\ge d_{n-1}\ge \cdots \ge d_{n-k+2}\ge 0\hspace{1cm}\text{and} \hspace{1cm} d_{n-k}\le \cdots \le d_2\le d_1.$$

   Then, $e_\nu^{k}$ is a product of $e_d^{(k)}$'s, each of which is a sum of products of no more than $k-1$ $x's$ or no more than $n-k$ $y$'s.  Since each factor is symmetric, we can take the highest possible indices for the $x$'s from each factor and the lowest from $y$'s, and multiply them by $\mathbf{xy}_S^S$ to obtain a new leading term that still does not vanish; there will still be at most $k-1$ nonzero $x$ variables and $n-k$ nonzero $y$ variables and no $x_iy_i$.  We also cannot obtain this term by other products of terms from the factors in $e_\nu^{k}$ by any permuted version of $\mathbf{xy}_S^S$, because any other option would have the $x$ degrees or $y$ degrees be lower in lexicographic order (and reverse lexicographic order, respectively).  
   
   Therefore, the leading term in question does not vanish in the polynomial ring, and since it does not vanish in the quotient $\mathcal{P}_n^{(k)}$, we are done.
\end{proof}

We can now conclude the following.

\begin{prop}\label{prop:intermediate-basis}
    The set of polynomials $\{F_T^S\cdot e_\nu^{(k)}\}$, ranging over all pairs $(T,S)$ of SYT's of the same shape and all $\nu$ with $\nu_1\le n$, forms a (higher Specht) basis of $\mathcal{P}_n^{(k)}$.
\end{prop}

\begin{proof}
    We have shown they are independent; it now suffices to compare the number of polynomials in each degree to a known basis to show there are enough polynomials.  In \cite{HookBases} it was shown that the set $\{a_\pi^{(k)}e_\nu^{(k)}\}$ for $\pi\in S_n$ and $\nu$ as above forms a basis for $\mathcal{P}_n^{(k)}$, where  $a_\pi^{(k)}$ are the \textit{descent monomial basis} of $\DR_\mu$.  Since they form a basis of $\DR_\mu$, and we know that the polynomials $F_T^S$ agree in degree with any basis of $\DR_\mu$ (Lemma \ref{lem:degrees}), the result follows.
\end{proof}

We can finally conclude our main result, which we restate in simpler language here. 

\begin{MainTheorem}
    The set $\{F_T^S\}$ is a higher Specht basis for $\DR_\mu$.
\end{MainTheorem}

\begin{proof}
    By Proposition \ref{prop:intermediate-basis}, the images of the polynomials $F_T^S\cdot e_\nu^{(k)}$ in the further quotient $\DR_\mu$ form a spanning set.  To see that they are a basis, we note again that there are the correct number of elements in each degree by Lemma \ref{lem:degrees}.  By Lemmas \ref{lem:MS} and \ref{lem:still-nonzero}, they are a higher Specht basis. 
\end{proof}

\section{Future directions and observations}\label{sec:future}

One corollary of Theorem \ref{thm:main} is that we can express the graded Frobenius series for hook shapes $\mu=(n-k+1,1^{k-1})$ as
$$\grFrob_{q,t}(\DR_\mu)=\sum_{S\in \SYT(n)} q^{\cc_\mu(S)}t^{\cc'_\mu(S)}s_{\sh(T)}.$$
It is a more general open problem than Problem \ref{prob:nfact} to find a combinatorial Schur expansion for the Frobenius series of $\DR_\mu$, and one possible route would be to generalize these new $\cc_\mu$ and $\cc'_\mu$ statistics to more general shapes $\mu$.  One also could attempt to generalize the higher Specht basis directly, since any higher Specht construction would automatically give a formula for the Frobenius character.

As for the diagonal coinvariant ring (Problem \ref{prob:basis-DR}), it would be interesting to see if a higher Specht basis could be constructed for them by combining the general tools from Section \ref{sec:general} with the new basis of Carlsson and Oblomkov \cite{CarlssonOblomkov}, or by interpreting parking functions as certain pairs of tableaux.  
In addition, Garsia and Zabrocki \cite{garsiaZabrocki} have established another basis for the diagonal harmonic alternants, essentially giving the copies of $s_{(1^n)}$ in $\DR_n$ for all $n$, and there is a combinatorial formula due to Alfano \cite{Alfano} for the Frobenius character of subspace having degree $1$ in the $y$ variables.  These are both potential starting points for constructing (partial) higher Specht bases for $\DR_n$, and potentially resolving the question of finding a Schur positive expansion for $\grFrob_{q,t}(\DR_n)$.

To conclude, we provide higher Specht bases for $\DR_2$ and $\DR_3$.  The Shuffle theorem (conjectured in \cite{HHLRU}, proven in \cite{shuffle}) gives a combinatorial formula for the graded Frobenius series of $\DR_n$, which is also equal to a Macdonald eigenoperator $\nabla$ applied to the elementary symmetric function $e_n$.

\begin{figure}[t]
    \begin{center}
    \begin{tabular}{cc||cc}
    Frobenius term & Basis polynomials & Frobenius term & Basis polynomials \\\hline
    & & & \\
    $s_{(3)}$ & $1$ & $q^2s_{(2,1)}$ & $\varepsilon_{{\tiny \young(3,12)}}x_3x_2$ \\
     & & & $\varepsilon_{{\tiny \young(2,13)}}x_2x_3$\\
    $qs_{(2,1)}$ & $x_2-x_1$ & & \\
            & $x_3-x_1$ & $t^2s_{(2,1)}$ & $\varepsilon_{{\tiny \young(3,12)}}y_3y_2$\\
            & & & $\varepsilon_{{\tiny \young(2,13)}}y_2y_3$\\
    $ts_{(2,1)}$ & $y_2-y_1$ & & \\
            & $y_3-y_1$ & $q^3 s_{(3)}$ & $(x_3-x_2)(x_3-x_1)(x_2-x_1)$\\
            & & & \\
    $qts_{(1,1,1)}$ & $\varepsilon_{{\tiny \young(3,2,1)}}x_3y_1$ & $t^3s_{(3)}$& $(y_3-y_2)(y_3-y_1)(y_2-y_1)$ \\
    & & & \\
    $qts_{(2,1)}$ & $\varepsilon_{{\tiny \young(3,12)}}x_3y_2$ & $q^2t s_{(1,1,1)}$& $\varepsilon_{{\tiny \young(3,2,1)}}x_3^2y_1$\\
    & $\varepsilon_{{\tiny \young(2,13)}}x_2y_3$ & $qt^2 s_{(1,1,1)}$ & $\varepsilon_{{\tiny \young(3,2,1)}}x_3y_1^2$\\
\end{tabular}
\end{center}
    \caption{A set of higher Specht polynomials for $\DR_3$.}
    \label{fig:DR3}
\end{figure}

We used Sage to expand $\nabla e_2$ and $\nabla e_3$ in terms of Schur functions, giving us the decompositions of $\DR_2$ and $\DR_3$ into irreducibles.  Then, by testing that the following polynomials, which generate the irreducibles in the correct degrees, do not vanish in the ideal, we were able to construct the following bases.

\begin{example}
For $\DR_2$, a higher Specht basis is: $1, x_2-x_1, y_2-y_1$.
\end{example}

\begin{example}
    For $\DR_3$, its graded Frobenius series from the Shuffle theorem is $$\nabla e_3 = s_{(3)}+(q+t)s_{(2,1)}+(q^2+qt+t^2)s_{(2,1)}+qt s_{(1,1,1)}+(q^3+t^3+q^2t+t^2q)s_{(1,1,1)}.$$  We verified in Sage \cite{sage} that the set of polynomials in Figure \ref{fig:DR3} do not vanish modulo the ideal $I_n$ generated by the positive degree diagonal invariants.  Therefore, by noting that the degrees of each higher Specht polynomial match the Frobenius series, they form a higher Specht basis for $\DR_3$.
\end{example}

\printbibliography

\end{document}